\documentclass[leqno]{article}
\usepackage{amsfonts,amsmath,amssymb,amsthm}
\usepackage{tikz}
\usepackage{graphicx,epsfig}
\usepackage{checkend}
\usepackage{longtable,geometry}

\numberwithin{equation}{section}
\usepackage{mathptmx}
\usepackage{amssymb,amsfonts}
\usepackage{epstopdf}
\usepackage{booktabs}
\usepackage{placeins}
\usepackage{graphicx}
\usepackage{mathdots}
\usepackage{cleveref}
\usepackage{mathrsfs}
\usepackage{txfonts}
\usepackage{amsopn}

\usepackage{tikz}
\usepackage{pgf}
\usepackage{scalefnt}
\usepackage{pgfplots}
\newtheorem{thm}{Theorem}[section]
\newtheorem{cor}[thm]{Corollary}

\newtheorem{lem}[thm]{Lemma}

\theoremstyle{definition}
\newtheorem{defn}[thm]{Definition}

\theoremstyle{remark}

\begin{document}

\begin{center}  {\huge\textbf{Interface Development for the Nonlinear Degenerate Multidimensional Reaction-Diffusion Equations }}
\par \medskip\bigskip\end{center}
\begin{center} {\Large\textsc{Ugur G. Abdulla and Amna Abu Weden}}
\par \medskip\bigskip\end{center}
\begin{center} {\large\noindent \textsc{Department of Mathematics, Florida Institute of Technology, Melbourne, Florida 32901}}
\par \medskip\bigskip\end{center}
{\bf Abstract.} 
This paper presents a full classification of the short-time behavior of the interfaces in the Cauchy problem for the nonlinear second order degenerate parabolic PDE
\[ u_t-\Delta u^m +b u^\beta=0, \ x\in \mathbb{R}^N, 0<t<T \]
with nonnegative initial function $u_0$ such that
\[ supp~u_0 = \{|x|<R\}, \ u_0 \sim  C(R-|x|)^\alpha,  \quad{as} \ |x|\to R-0, \]
where $m>1, C,\alpha, \beta >0, b \in \mathbb{R}$. Interface surface $t=\eta(x)$ may shrink, expand or remain stationary depending on the relative strength of the diffusion and reaction terms near the boundary of support, expressed in terms of the parameters $m,\beta, \alpha, sign\ b$ and $C$. In all cases we prove explicit formula for the interface asymptotics, and local solution near the interface.

\section{Inrtroduction}
Consider the Cauchy problem for the Reaction-Diffusion equation:
\begin{equation}\label{CauchyProblem1}
Lu=u_t-\Delta u^m+bu^{\beta}=0,\ x\in \mathbb{R}^N, 0<t<T,
\end{equation}
\begin{equation}\label{CauchyProblem2}
u(x,0)=u_0(x), \quad x\in \mathbb{R}^N,
\end{equation}
where~$m>1,\ b\in \mathbb{R},\ \beta>0$. Equation \eqref{CauchyProblem1} is a nonlinear degenerate parabolic equation arising in various
applications in fluid mechanics, plasma physics, population dynamics etc. as a mathematical model
of nonlinear diffusion phenomena in the presence of absorption of energy \cite{Aris1975, Barenblatt1952,Kalashnikov1987,Bear1972}.
Assume that $u_0\in C(\mathbb{R}^N; \mathbb{R}^+)$ is radially symmetric with
\[
supp\ u_0=\overline{B_R}
\]
where $B_R:=\{x\in \mathbb{R}^N,\ |x|<R\}$,
and
\begin{equation}\label{CauchyProblem3}
 u_0(x)\sim C(R-|x|)^\alpha~as~|x|\rightarrow R-0
\end{equation}
for some $C>0, \alpha>0$. Typical example is
\begin{equation}\label{CauchyProblem4}
u_0(x)= C(R-|x|)^{\alpha}_+,\ x\in \mathbb{R}^N
\end{equation}
where $\kappa_+=\max\{\kappa;0\}$. Solution of the Cauchy Problem \eqref{CauchyProblem1},\eqref{CauchyProblem2} is understood in a weak sense (Definition~\ref{weaksolution}, Section~\ref{sec:prelim}). Furthermore, we will assume that $\beta\geq 1$ if $b<0$, which is essential to guarantee uniqueness of the solution. Weak solution possesses a finite speed of propagation property, meaning that it is compactly supported for any $t>0$ \cite{Barenblatt1952}. Boundary manifolds of the support of solution are called "free boundaries" or "interfaces". The main goal of this paper is to analyze short-time behavior of interfaces emerging from sphere $\partial B_R\times\{t=0\}$.

For all\ $x\in B_R$\ near the boundary define the interface surface as
\[
t=\eta_-(x)\coloneqq \sup\{\tau:u(x,t)>0,0<t<\tau\}.
\]
If\ $\eta_-(x)$~is defined and finite for all\ $x$\ such that\ $0<<|x|<R$\ and\ 
\begin{equation}\label{osmall}
\eta_-(x)=o(1) \quad\text{for} \ |x|\rightarrow R-0,
\end{equation}
then we say that the interface initially shrinks at\ $\partial B_R$. For all\ $x\in B^{c}_R=\{|x|>R\}$\ near the boundary define interface surface
\[
t=\eta_+(x)\coloneqq \inf\{\tau\geq 0:u(x,t)>0, \tau<t<\tau+\epsilon \quad\text{for some} \ \epsilon>0  \}.
\]
If\ $\eta_+(x)$\ is defined, positive and finite for all\ $x$\ such that\ $R<|x|<<+\infty$\ and satisfies \eqref{osmall}, 
 then we say that the interface initially expands at\ $\partial B_R$. If
 \[ supp~u(x,t)\equiv supp~u_0(x) \]
 for all\ $0\leq t\leq \delta,$\ for some\ $\delta>0,$\ then we say that interface remain stationary, or solution has a waiting time near the support of the initial function.

The goal of this paper is to present full classification of the existence and short time behavior of the interfaces $\eta_{\pm}$, and local solution near $\eta_{\pm}$ in terms of the parameters $m,\beta,b,C,\alpha$. 

The outline of the paper is as follows. In Section~\eqref{sec:mainresult} we formulate the main results. Theorems~\ref{theorem1}-\ref{theorem5} of Section~~\ref{sec:mainresult} present full classification of the short-time behavior and asymptotics of the interfaces with respect to relative strength of diffusion versus reaction/absorption expressed in respective four regions of the parameter space $(\alpha,\beta)$.  Some essential technical details of the main results are outlined in Section~\ref{sec: details of the main results}. In Section~\ref{sec:prelim} 
we present brief literature review and prove important asymptotic properties of local solutions along the special interface-type manifolds by using rescaling and application of the general theory of nonlinear degenerate parabolic equations. Finally, in Section~\ref{proofs} we prove the main results by using asymptotic estimations of Section~\ref{sec:prelim} 
and by constructing local super- and subsolutions based on the special comparison theorems in general non-cylindrical domains with irregular and characteristic boundary manifolds.

\section{Description of main results.}\label{sec:mainresult}
Throughout this section we assume that $u$ is a unique weak solution of the CP \eqref{CauchyProblem1}-\eqref{CauchyProblem2}. There are four different subcases, as shown in Figure 1. The main results are outlined below in Theorems~\ref{theorem1}-\ref{theorem5} corresponding directly to the cases (1), (2), (3) and (4).

 \begin{figure}
		\begin{center}
\begin{tikzpicture}[xscale=1.5, yscale=0.5]
\draw[->, thick] (0,0)--(2.5,0);
\draw[->,thick] (0,0)--(0,6);
\node[below left] at (0,0) {$0$};
\node[fill, shape=circle, label=180:{$1$}, inner sep=1pt] at (0,2) {};
\node[fill, shape=circle, label=180:{$m$}, inner sep=1pt] at (0,4) {};
\node[fill, shape=circle, label=-90:{$\displaystyle  \frac{2}{m-1}$}, inner sep=1pt] at (.71,0) {};
\node[left] at (0,6.4) {\scalefont{1.5} $\beta$};
\node[below] at (2.6,0) {\scalefont{1.5} $\alpha$};
\draw[domain=0:2.5, smooth] plot ({1/(3-\x)},\x);
\node at (2.3,2.9) {$\alpha =2/(m-\beta)$};
\node at (0.3,4.6) {\scalefont{1.1} $(1)$};
\node (2) at (0.2,1.7) {(2)};
\node (3) at (.9,.6) {(3)};
\node (4a) at (0.5,2.5) {(4a)};
\node (4b) at (1.5,.8) {(4b)};
\node (4c) at (1.8,2) {(4c)};
\node (4d) at (1.2,3.9) {(4d)};
\node (*) at (0.5,.6) {};
\node (**) at (.71,1.5) {};
\node (***) at (1.4,1.74) {};
\draw[ ->][ left]
  (2) edge (*);
 \draw[ ->][ left]
  (4a) edge (**);
   \draw[ ->][ left]
  (4b) edge (***);
\draw (.71,1.6)--(.71,5);
\draw (.7,1.62)--(2,1.62);
\draw ( (0.5,1) {};
\end{tikzpicture}
\caption{Classification of different cases in the ($\alpha$,$\beta$) plane for interface development in problem \eqref{CauchyProblem1}, \eqref{CauchyProblem2}, \eqref{CauchyProblem3}.}
		\label{fig:1}
		\end{center}
		\end{figure}
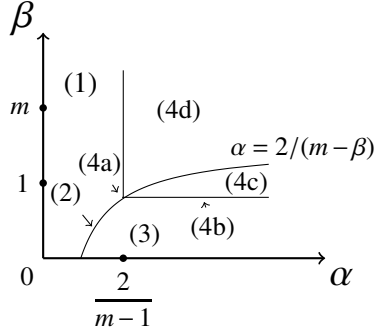
\begin{thm}\label{theorem1}
If $\alpha <\frac{2}{m-min\{1,\beta\}}$, then the interface initially expands and
\begin{equation}\label{eta+}
\eta_+(x)\sim \Big ( \frac{R-|x|}{\xi_*}\Big )^{2+\alpha(1-m)}~\mathrm{as}\  \ |x|\to R+,
\end{equation}
where \ $ \xi_*=\xi_*(C,\alpha,m)<0$. For arbitrary $\xi_*<\rho<0$, there exists a positive number $f(\rho)$ depending on $C,m,\alpha$ such that
\begin{equation}\label{localsolution1}
u \ \Big |_{|x|=R-\rho t^{\frac{1}{2+\alpha(1-m)}}}\sim f(\rho) t^{\frac{\alpha}{2+\alpha(1-m)}}~\mathrm{as}\ t\downarrow 0.
\end{equation}
\end{thm}
\begin{thm}\label{theorem3}
Let $b>0,$\ $0<\beta<1,$\ $\alpha=\frac{2}{m-\beta}, m+\beta\neq 0$ and
\[C_*=\Big\{\frac{b(m-\beta)^2}{2m(m+\beta)}\Big\}^{\frac{1}{m-\beta}}\]
If $C>C_*$ then interface initially expands and
\begin{equation}\label{eta_+1}
\eta_+(x)\sim \Big(\frac{|x|-R}{\zeta_+}\Big)^{\frac{2(1-\beta)}{m-\beta}}~\mathrm{as}\  \ |x|\rightarrow R+,
\end{equation}
while if $C<C_*$ then interface initially shrinks and
\begin{equation}\label{eta_-1}
\eta_-(x)\sim \Big(\frac{R-|x|}{\zeta_-}\Big)^{\frac{2(1-\beta)}{m-\beta}}~\mathrm{as}\  \ |x|\rightarrow R-,
\end{equation}
where $\zeta_+\in[\zeta_1;\zeta_*]$,  $\zeta_-\in[\zeta_2;\zeta_*]$(see Appendix for constants $\zeta_1,\zeta_2$), $\zeta_*=\zeta_*(C,m,\beta,b)\gtrless 0$ \ according to as\ $C\lessgtr C_*$. 
For arbitrary $\rho>\zeta_*$ there exists $h(\rho)>0$ such that
\begin{equation}\label{asympsolution2}
u(x,t) \ \Big |_{|x|=R-\rho t^{\frac{m-\beta}{2(1-\beta)}}}\sim h(\rho)t^{\frac{1}{1-\beta}}\qquad~\mathtt{as}\ \ t\downarrow 0.
\end{equation}
\end{thm}
\begin{cor}\label{corollary}
If conditions of Theorem~\ref{theorem3} are satisfied and $m+\beta=2$, then claims \eqref{eta_+1},\eqref{eta_-1},\eqref{asympsolution2}
are valid with
\begin{equation}\label{m+beta=2}
\zeta_+=\zeta_-= \zeta_*=b(1-\beta)C^{\beta-1}\Big(1-\big(C/C_*\big)^{2(1-\beta)}\Big), \ h(\rho)=C(\rho-\zeta_*)_+.
\end{equation}
\end{cor}
\begin{thm}\label{theorem4}
Let $b>0,0<\beta<1, \alpha>\frac{2}{m-\beta}$. Then interface initially shrinks and
\begin{equation}\label{shrinkinginterface}
\eta_-(x)\sim \Big(\frac{R-|x|}{l_*}\Big)^{\alpha(1-\beta)}~\mathrm{as}\  \ |x|\rightarrow R-,
\end{equation}
where $l_*=C^{-{\frac{1}{\alpha}}}(b(1-\beta))^{\frac{1}{\alpha(1-\beta)}}$. For $\forall l>l_*$ we have
\begin{equation}\label{shrinkingsolution}
u \ \Big |_{|x|=R-lt^{\frac{1}{\alpha(1-\beta}}}\sim \{C^{1-\beta}l^{\alpha(1-\beta)}-b(1-\beta)\}^{\frac{1}{1-\beta}}t^{\frac{1}{1-\beta}}\qquad~\mathtt{as}\ \  t\downarrow 0.
\end{equation}
\end{thm}
\begin{thm}\label{theorem5}
If\ $\beta\geq 1,$\ $\alpha\geq \frac{2}{m-1,}$\ then the interface initially remains stationary.
\end{thm}
\section{Technical Details of the Main Results}\label{sec: details of the main results} 
In this section we outline some essential details of the main results described in Theorems~\ref{theorem1}-\ref{theorem5} of \cref{sec:mainresult}. \\
{\it Technical details of Theorem~\ref{theorem1}}: Precise values of the constant $\xi_*$ and the function $f$ are associated with the one-dimensional Cauchy Problem \cite{Abdulla1}
\begin{gather}
w_t=(w^m)_{yy},\qquad y\in \mathrm{R},\ 0<t<+\infty\label{1ddiffusion}\\
w(y,0)= C(y)_+^\alpha,\qquad  y\in \mathrm{R}\label{1ddiffusionic}
\end{gather}
There exists a unique solution of the problem \eqref{1ddiffusion},\eqref{1ddiffusionic}, which is of self-similar form
\begin{equation}\label{selfsimilarform}
w(y,t) = t^{\frac{\alpha}{2+\alpha(1-m)}}f(\xi),\ \mathrm{where} \  \xi = \frac{y}{t^{\frac{1}{\alpha(1-m)+2}}},
\end{equation}
and the shape function $f$ solves nonlinear ODE problem
\begin{equation}\label{selfsimilarODE}
\left\{
\begin{array}{ll}
\frac{d^2 f^m}{d\xi^2} + \frac{1}{2+\alpha(1-m)}\xi \frac{df}{d\xi}- \frac{\alpha}{2+\alpha(1-m)}f = 0,\quad \xi \in \mathbb{R}\\
f \sim C\xi^\alpha\qquad as~\xi \rightarrow +\infty,\ f(\xi) = o(|\xi|^\alpha)~~as\quad \xi \downarrow -\infty,
\end{array} \right.
\end{equation}
with finite interface $\xi_*=\xi_*(C,\alpha,m)<0$\ such that
\begin{equation}\label{xi*}
f(\xi)>0,\ \xi_*<\xi<+\infty;\ f(\xi)\equiv 0,\ \xi<\xi_*
\end{equation}
Through rescaling one can find dependence of $f$ and $\xi_*$ on $C$ \cite{Abdulla1}:
\begin{equation}\label{CP7}
f(\rho)=C^{\frac{2}{2-\alpha(m-1)}}f_1(C^{\frac{m-1}{2-\alpha(m-1)}}\rho)
\end{equation}
\begin{equation}\label{CP8}
f_1(\rho)=w_1(\rho,1),\ \xi_*^{'}=inf\{x:f_1(\rho)>0\}<0
\end{equation}
\begin{equation}\label{CP6}
\xi_*=C^{\frac{m-1}{2-\alpha(m-1)}}\xi_*^{'}
\end{equation}
where $w_1$ and $f_1$ are solutions of \eqref{1ddiffusion},\eqref{1ddiffusionic}, and \eqref{selfsimilarODE}, respectively, with the constant $C=1$; \ $\xi_*^{'}$\ is a negative number depending on\ $m~\mathrm{and}\ \alpha$\ only. \\
{\it Technical details of Theorem~\ref{theorem3}}: Precise values of the constant $\zeta_*$ and the function $h$ are associated with the one-dimensional Cauchy Problem
\begin{gather}
w_t=w^m_{yy}-bw^\beta,\qquad y\in \mathrm{R},\ 0<t<+\infty\label{1dreactiondiffusion}\\
w_0(y)= C(y)_+^{\frac{2}{m-\beta}},\qquad  y\in \mathrm{R}\label{1dreactiondiffusionic}
\end{gather}
There exists a unique solution of the problem \eqref{1dreactiondiffusion},\eqref{1dreactiondiffusionic}, which is of self-similar form
\begin{equation}\label{selfsimilarform1}
w(y,t) = t^{\frac{1}{1-\beta}}h(\zeta),\ \mathrm{where} \  \zeta = \frac{y}{t^{\frac{m-\beta}{2(1-\beta)}}},
\end{equation}
and the shape function $h$ solves nonlinear ODE problem
\begin{equation}\label{selfsimilarODE2}
\left\{
  \begin{array}{ll}
\frac{1}{1-\beta}h-\frac{m-\beta}{2(1-\beta)}\zeta h^{'}-(h^m)^{''}+bh^{\beta}=0,\qquad\zeta \in \mathbb{R}. \\
h(\zeta)\sim C\zeta^{\frac{2}{m-\beta}}\quad\text{as}~\zeta \uparrow +\infty,\ h(\zeta)= o(|\zeta|^{\frac{2}{m-\beta}})~\text{as}~\zeta \downarrow -\infty. \\
 \end{array}
 \right.
 \end{equation}
There exists a finite interface $\zeta_*$ such that $\zeta_*=\zeta_*(C,m,\beta,b)$\ such that $\zeta_*\gtrless 0$ \ according to as\ $C\lessgtr C_*$, and
\begin{equation}\label{zeta*}
h(\zeta)>0,\ \zeta_*<\zeta<+\infty;\ h(\zeta)\equiv 0, \ \zeta \leq \zeta_*
\end{equation}
In the special case $m+\beta=2$ as in Corollary~\ref{corollary}, the explicit solution of the problems \eqref{1dreactiondiffusion}-\eqref{1dreactiondiffusionic} and \eqref{selfsimilarODE2}
are 
\begin{equation}\label{explicitsol}
w(y,t)=C(y-\zeta_*t)_+^{\frac{1}{1-\beta}}, \ h(\zeta)=C(\zeta-\zeta_*)_+^{\frac{1}{1-\beta}}
\end{equation}
with $\zeta_*$ defined in \eqref{m+beta=2}.\\
{\it Technical details of Theorem~\ref{theorem5}}: There are four subcases.\\
(5a)\ If\ $\beta=1,\ \alpha=\frac{2}{m-1}$\ then $\forall \epsilon >0\ \exists R_\epsilon \in (0,R)$ and $\delta_\epsilon >0$ such that 
\begin{gather}
(C-\epsilon)(R-|x|)^{\frac{2}{m-1}}_+e^{-bt}
\leq u \leq \nonumber\\
(C+\epsilon)(R-|x|)^{\frac{2}{m-1}}_+ e^{-bt}\Big(1-(C/\bar{C})^{m-1}b^{-1}(1-e^{-b(m-1)t})\Big)^{\frac{1}{1-m}}
\label{stationary1}
\end{gather}
for\ $R_\epsilon\leq |x|<+\infty,\  0\leq t\leq \delta_\epsilon,$\ and
\begin{equation}\label{CP13}
    T=\left\{
                \begin{array}{ll}
                 +\infty,\qquad\qquad\qquad if\qquad b\geq (C/\bar{C})^{m-1}\\
                  \frac{\ln\big(1-b(C/\bar{C})^{m-1}\big)}{b(1-m)},\quad if\ -\infty<b<(C/\bar{C})^{m-1}, \\
                  \end{array}
              \right.
 \end{equation}
 \begin{equation}\label{CP14}
 \bar{C}=\Big[\frac{(m-1)^2}{2m(m+1)}\Big]^{\frac{1}{m-1}}.
 \end{equation}
 (5b)\ If\  $\beta=1,\ \alpha>\frac{2}{m-1},$ then $\forall \epsilon>0,\ \exists R_\epsilon>0,\ and\ \delta_\epsilon>0$\ such that 
\begin{gather}
(C-\epsilon)(R-|x|)^{\alpha}_+e^{-bt}
\leq u \leq (C+\epsilon)(R-|x|)^{\alpha}_+ e^{-bt}\times\nonumber\\
\Big(1-\epsilon(b(m-1))^{-1}(1-e^{-b(m-1)t})\Big)^{\frac{1}{1-m}}, \ |x|\geq R_\epsilon, \ 0\leq t \leq \delta_\epsilon.\label{stationary2}
\end{gather}
(5c)\ If\ $1<\beta<m,\ \alpha\geq \frac{2}{m-\beta},$\ then $\forall \epsilon>0,\ \exists R_\epsilon>0,\ and\ \delta_\epsilon>0$\ such that 
\begin{equation}\label{stationary3}
g_{-\epsilon}\leq u(x,t)\leq g_{\epsilon},\quad |x|\geq R_\epsilon,\quad 0\leq t\leq \delta_\epsilon,
\end{equation}
\begin{equation}
g_{\epsilon}(x,t)=\left\{
                \begin{array}{ll}
[(C+\epsilon)^{1-\beta}(R-|x|)^{\alpha(1-\beta)}+b(\beta-1) (1-d_{\epsilon})t]^{\frac{1}{1-\beta}},\ \ R_\epsilon\leq |x|\leq R\\
0,\qquad\qquad\qquad |x|\geq R,
                  \end{array}
              \right.
\end{equation}
\begin{equation}\label{CP15}
d_{\epsilon}=\left\{
                \begin{array}{ll}
                  \epsilon\ sign\ b,\qquad \qquad \qquad  if  \ \alpha>\frac{2}{m-\beta}\\
                  \Big(\big(C+\epsilon/\bar{C}\big)^{m-\beta}+\epsilon \Big)sign\ b\ \quad if\ \alpha=\frac{2}{m-\beta} . \\
                  \end{array}
              \right.
\end{equation}
(5d)\ Let either\ $1<\beta<m,\frac{2}{m-1}\leq \alpha< \frac{2}{m-\beta}$\ or\ $\beta\geq m.\  \alpha \geq \frac{2}{m-1}$.\ If\ $\alpha= \frac{2}{m-1}$, then for arbitrary small\ $\epsilon>0\ \exists R_\epsilon>0,\ and\ \delta_\epsilon>0$\ such that for  $|x|\geq R_\epsilon,\quad 0\leq t\leq \delta_\epsilon$ we have
\begin{equation}\label{stationary4}
(C-\epsilon)(R-|x|)^{\alpha}_+\big(1-\gamma_{-\epsilon} t\big)^{\frac{1}{1-m}}\leq u(x,t)\leq (C+\epsilon)(R-|x|)^{\frac{2}{m-1}}\big(1-\gamma_\epsilon t\big)^{\frac{1}{1-m}},
\end{equation}
where
\begin{equation}\label{CP16}
\gamma_\epsilon=\Big[\frac{2m(m+1)(C+\epsilon)^{m-1}}{m-1}\Big]+\epsilon
\end{equation}
If\ $\alpha >\frac{2}{m-1}$\ then for arbitrary small\ $\epsilon>0\ \exists R_\epsilon>0,\ and\ \delta_\epsilon>0$\ such that
\begin{equation}\label{stationary4"}
(C-\epsilon)(R-|x|)^{\alpha}_+\leq u(x,t)\leq  (C+\epsilon)(R-|x|)^{\alpha}_+\big(1-\epsilon t\big)^{\frac{1}{1-m}}, \ |x|\geq R_\epsilon, \ 0\leq t\leq \delta_\epsilon.
\end{equation}

\section{ Preliminary results.} \label{sec:prelim}
Solution of the Cauchy problem \eqref{CauchyProblem1},\eqref{CauchyProblem2} is understood in the following weak sense:
\begin{defn}\label{weaksolution}
The function $u(x,t)$ is said to be a solution (respectively, super- or subsolution) of the Cauchy Problem \eqref{CauchyProblem1},\eqref{CauchyProblem2}, if
\begin{itemize}
\item $u$ is nonnegative and continuous in $\mathbb{R}^N\times [0,T)$, locally H\"{o}lder continuous in $\mathbb{R}^N\times(0,T)$, satisfying \eqref{CauchyProblem2}
(respectively, satisfying \eqref{CauchyProblem2} with $=$ replaced by $\geq$ or $\leq$),
\item for any $t_0$, $t_1$ such that $0<t_0<t_1<T$ and for any bounded domain $\Omega\subset\mathbb{R}^N$ with smooth boundary $\partial\Omega$ the following integral 
identity holds:
\begin{gather}
\int\limits_{\Omega\times\{t=t_1\}}ufdx=\int\limits_{\Omega\times\{t=t_0\}}ufdx+\int\limits_{\Omega\times\{t_0<t<t_1\}}(uf_t+u^m \Delta f - bu^\beta f)dx dt\nonumber\\
-\int\limits_{\partial\Omega\times(t_1,t_2)}u^m \frac{\partial f}{\partial \nu} dx dt, \label{integralidentity}
\end{gather}
(respectively, \eqref{integralidentity} holds with $=$ replaced by $\geq$ or $\leq$), where $f\in C_{x,t}^{2,1}(\overline{\Omega})$ is an arbitrary function (respectively, nonnegative function)
that equals to zero on $\partial\Omega\times[t_0,t_1]$ and $\nu$ is the outward-directed normal vector to $\partial\Omega$.
\end{itemize}
\end{defn}
Prelude of the theory of second order nonlinear degenerate parabolic equations are the papers \cite{Barenblatt1952,zeldovich}, which revealed the property of finite speed of propagation  of weak solutions due to implicit degeneration of the PDE. Importance of the analysis of the interfaces are twofold. First, this indicates more relevance for the physical applications in comparison with linear diffusion with infinite speed of propagation property. Second, non-smoothness of the weak solutions are concentrated primarily along to interfaces, that is to say along the zero level set of the solution where uniform parabolicity is violated. Mathematical theory of the second order nonlinear degenerate parabolic PDEs begins with the work \cite{Oleinik1958}. Currently there is a well established theory of well-posedness of main boundary value problems, and local regularity properties of weak solutions \cite{Aronson1981,Aronson1982,Caffarelli1979,Dibe-Sv,dibenedetto1983continuity,alt1983quasilinear,Lacey1982,aronson1983initially,kersner1980degenerate,BCP,Caffarelli1990,Caffarelli1987,Vazquez1984,Kalashnikov1987,Samarskii1987,otto,ADS}. Without any ambition to present full survey of outstanding contributions by many mathematicians, we refer to \cite{Dibe-Sv,Vazquez2} which outline the modern well established theory and contain extensive list of references. General theory of boundary value problems in non-cylindrical domains with non-smooth boundary manifolds under minimal regularity assumptions on the boundaries is developed in \cite{abdulla2001dirichlet,abdulla2005well,abdulla2006reaction}. 
In particular general theory in non-cylindrical non-smooth domains was motivated by the problem about the evolution of interfaces. To present complete classification of the development of interfaces it is essential to apply general theory of boundary-value problems in non-cylindrical domains with boundary surfaces which has the same kind  of behaviour as the interface. In many cases this may be nonsmooth and characteristic. 

We now make precise the meaning of the solution to Dirichlet problem (DP) in general domains. Let $\Omega$ be an open subset of $\mathbb{R}^{N+1},N \geq 2$. Let the boundary $\partial \Omega$ of $\Omega$ consist of the closure of a domain $B\Omega$ lying on $t = 0$, a domain $D\Omega$ lying on $t = T \in  (0,\infty)$ and a (not necessarily connected) manifold $S\Omega$ lying in the strip $0 < t \leq T$. Assume that $\Omega(\tau):=\Omega\cap\{t=\tau\}\neq\emptyset$ for $t \in (0, T )$. 

The set $\mathcal{P}\Omega = \overline{B\Omega}\cup S\Omega$ is called a parabolic boundary of $\Omega$. The class of domains with described structure is denoted by $\mathcal{D}_{0,T}$ . Let $\Omega\in \mathcal{D}_{0,T}$ be given and let $ \psi$ be an arbitrary continuous non-negative function defined on $\mathcal{P}\Omega$. DP consists of finding a solution to equation\eqref{CauchyProblem1} in $\Omega\cup D\Omega$ satisfying the initial-boundary condition
\begin{equation}\label{bddfunction}
u=\psi ~on~\mathcal{P}\Omega 
\end{equation}
\begin{defn}[Weak Solution of the DP](\cite{abdulla2005well,abdulla2006reaction}) \label{def: weak soln} 
We say that a function~$u(x,t)$~is  a solution (resp., super- or subsolution) of DP \eqref{CauchyProblem1},\eqref{bddfunction} if
\begin{itemize}
\item  $u$ is nonnegative, bounded and continuous in $\overline{\Omega}$, and  locally H\"{o}lder continuous in $\Omega \cup D\Omega$ satisfying \eqref{bddfunction} (respectively satisfying \eqref{bddfunction} with $=$ replaced by $\geq$ or $\leq$)
\item  for any~$t_0,t_1$such that~$0<t_0<t_1<T$, and for any domain  $\overline{\Omega}_1\in \mathcal{D}_{t_0,t_1}$~such that
$\overline{\Omega}_1 \subset \Omega\cup D\Omega$ and $\partial B\Omega_1, \partial D\Omega_1, S\Omega_1$ being sufficiently smooth manifolds, the following
integral identity holds:
\begin{equation}\label{int/identity}
\int_{D\Omega_1}ufdx=\int_{B\Omega_1}ufdx+\int_{S\Omega_1}(uf_t+u^m\Delta f)dxdt-\int_{S\Omega_1}u^m\frac{\partial f}{\partial \nu}dxdt
\end{equation}
(respectively \eqref{int/identity} holds with $=$ replaced by $\geq$ or $\leq$, where $f \in C^{2,1}_{x,t}(\overline{\Omega}_1)$ is an
 arbitrary function (respectively non-negative function) that equals zero on $S\Omega_1$
and $\nu$ is the outward-directed normal vector to $\Omega_1(t)$ at $(x,t) \in S\Omega_1$.
\end{itemize}
\end{defn}
In \cite{abdulla2001dirichlet,abdulla2005well,abdulla2006reaction} existence, boundary regularity, uniqueness and comparison theorems for the DP are proved under minimal pointwise assumption on the local modulus of lower semicontinuity of the boundary manifold $S\Omega$ (see Assumption $\mathcal{A}$ and Assumption $\mathcal{M}$ in \cite{abdulla2001dirichlet,abdulla2005well,abdulla2006reaction}). In particular, the following comparison theorem will be of essential use in this paper:
\begin{thm}\label{comparison} (\cite{abdulla2005well,abdulla2006reaction}). Let $u$ be a solution of DP and let $g$ be a supersolution (respectively subsolution) of DP. Assume that the assumption Assumption $\mathcal{A}$ and Assumption $\mathcal{M}$ of \cite{abdulla2005well} are satisfied. Then $u \leq$ (respectively $\geq$) $g$ in $\Omega$.
\end{thm}
The initial development of interfaces and local structure of solutions near the interfaces is very well understood in the one dimensional case. Full classification of evolution of interfaces 
and local behavior of solutions near the interfaces for the problem \eqref{CauchyProblem1}-\eqref{CauchyProblem3} with space dimension $N=1$ was presented in \cite{Abdulla1}
for slow diffusion case ($m>1$), and in \cite{abdulla2002evolution} for the fast diffusion case ($m=1$). The results and methods of \cite{Abdulla1,abdulla2002evolution} are extended to solve interface problem for $p$-Laplacian type reaction-diffusion equations in \cite{AbdullaJeli1,AbdullaJeli2}, and for the reaction-diffusion equations with double degenerate diffusion in \cite{AbdullaPrinkey1}. The method of the proof developed in \cite{Abdulla1,abdulla2002evolution} is based on rescaling and application of the one-dimensional theory of reaction-diffusion equations in general non-cylindrical domains with non-smooth boundary curves developed in \cite{abdulla2000reaction,abdulla2000reaction1}. Sharp asymptotic estimates for the interfaces and local solutions of the Dirichlet problem for the equation \eqref{CauchyProblem1} in bounded cylindrical domains domains was proved in \cite{abdulla1999local}. 
Estimation for the interfaces via energy methods is pursued in \cite{ADS}.

In the following three lemmas we establish asymptotic properties of the solution to the Cauchy problem \eqref{CauchyProblem1}-\eqref{CauchyProblem3} based on the scaling laws corresponding to the PDE \eqref{CauchyProblem1}.
\begin{lem}\label{lemma1}
Let\ $u$\ solves\ CP \eqref{CauchyProblem1}-\eqref{CauchyProblem3} \ with\ $b=0,\ $ and one of the following conditions is satisfied.\\
(i)\ $b=0,\ 0<\alpha<\frac{2}{m-1}$\\
(ii)\ $b>0,\ 0<\beta<1,\ 0<\alpha<\frac{2}{m-\beta}$\\
(iii)\ $b\neq 0,\ \beta\geq1, 0<\alpha<\frac{2}{m-1}$\\
Then~u(x,t)~satisfies \eqref{localsolution1}.
\end{lem}
\begin{proof}[Proof] (i)\  First consider the global case \eqref{CauchyProblem4}. Change the variable $y=x+\bar{x}$ with $\bar{x}=(R,0,....,0)$.
Function $v(y,t)=u(y-\bar{x},t)$ solves the problem
\begin{gather}
v_t(y,t)-\Delta v^{m}(y,t)=0,~y\in\mathbb{R}^N,~ t>0\label{transformedequation}\\
v(y,0)=C(R-|y-\bar{x}|)_+^{\alpha},\ y\in \mathbb{R}^N\label{globalic}
\end{gather}
Since nonlinear diffusion equaton is invariant under the scaling
\[ y \rightarrow k^{-\frac{1}{\alpha}}y, \ \ t\rightarrow k^\frac{\alpha(m-1)-2}{\alpha}t \]
rescaled function 
\[v_k(y,t)=kv(k^{-\frac{1}{\alpha}}y,k^\frac{\alpha(m-1)-2}{\alpha}t)\]
solves the problem
\[\left\{
  \begin{array}{ll}
w_t(y,t)-\Delta w^m(y,t)=0\quad y\in \mathbb{R}^N~t>0\\
w(y,0)=C(k^{\frac{1}{\alpha}}R-|y-k^{\frac{1}{\alpha}}\bar{x}|)^{\alpha}_+\quad y\in \mathbb{R}^N.\\
 \end{array}
 \right.\]
 Since
  \begin{equation}\label{limitu0}
    \lim_{k\rightarrow \infty}C\{k^{\frac{1}{\alpha}}R-|y-k^{\frac{1}{\alpha}}\bar{x}|\}^{\alpha}_+=C(y_1)^{\alpha}_+,
    \end{equation}
the limit function solves the CP \eqref{CauchyProblem1},\eqref{CauchyProblem2} with $b=0$ and $u_0(x)=C(x_1)^\alpha_+$.    
Due to uniqueness of the solution to the CP (\cite{BCP}), the latter coincides with the solution of the 1D CP  \eqref{1ddiffusion},\eqref{1ddiffusionic}, which is of self-similar form \eqref{selfsimilarform} with the shape function $f(\xi)$ solving nonlinear ODE problem \eqref{selfsimilarODE} and having finite interface $\xi_*<0$. Therefore, we have
\[\lim_{k \rightarrow \infty}k v(k^{-\frac{1}{\alpha}}y,k^\frac{\alpha(m-1)-2}{\alpha}t)=t^{\frac{\alpha}{2+\alpha(1-m)}}f\Big(y_1t^{-\frac{1}{2+\alpha(1-m)}}\Big).\]
By choosing $y_2=.....=y_n=0$, $k^{\frac{\alpha(m-1)-2}{\alpha}}t=\tau$, $y_1=\rho t^{\frac{1}{2+\alpha(1-m)}}$, $\rho>\xi_*$, we have
\begin{equation}\label{pre6}
u(x,t) \arrowvert _{\Gamma_{\rho}}\sim f(\rho)  t^{\frac{\alpha}{2+\alpha(1-m)}}~as\ t\downarrow 0,
\end{equation}
where,
\[ \Gamma_{\rho}=\{(x,t):x_1=-R+\rho t^{\frac{1}{2+\alpha(1-m)}},\ x_2=.........=x_n=0,\ t\geq 0\} \]
Since the initial condition is radially symmetric, the solution of the CP is radially symmetric for any fixed $t>0$, and therefore, from \eqref{pre6}, \eqref{localsolution1} follows. 
Equivalently, for all $\rho$ with $\xi_*<\rho<0$ we have 
\begin{equation}\label{equiv 6}
u \ \Big |_{ t=\eta(x)=\Big (\frac{R-|x|}{\rho}\Big )^{2+\alpha(1-m)}}\sim f(\rho) t^{\frac{\alpha}{2+\alpha(1-m)}}~\quad\text{as} \ |x|\rightarrow R+.
\end{equation}
If $u_0(x)$ satisfies \eqref{CauchyProblem3}, 
then for arbitrary $\epsilon>0$ $\exists 0<R_\epsilon<R$ such that
\begin{equation}\label{local1}
 \Big(C-\epsilon/2\Big)(R-|x|)_+^{\alpha}\leq u_0(x) \leq \Big(C+\epsilon/2\Big)(R-|x|)_+^{\alpha}, \ |x|\geq R_\epsilon
\end{equation}
Let $u_{\pm\epsilon}$ be a solution of the CP \eqref{CauchyProblem1},\eqref{CauchyProblem2} with initial function
$(C\pm\epsilon )(R-|x|)_+^{\alpha}$.
Since the solution of CP \eqref{CauchyProblem1},\eqref{CauchyProblem2} is continuous, there exists a~$\delta_\epsilon>0$~such that 
\begin{equation}\label{local2}
u_{-\epsilon}(x,t)\leq u(x,t)\leq u_{+\epsilon}(x,t),\quad |x|=R_\epsilon, 0\leq t\leq \delta_\epsilon
\end{equation}
From (\ref{local1}),(\ref{local2}) and a comparison Theorem~\ref{comparison} applied in $\{|x|>R_\epsilon, 0<t\leq\delta_\epsilon\}$ we have
\begin{equation}\label{local3}
u_{-\epsilon}\leq u \leq u_\epsilon \quad \mathrm{for}\quad |x|>R_\epsilon, 0<t\leq \delta_\epsilon
\end{equation}
As we have already proved for all $\rho$ such that $\xi_*(C\pm\epsilon)<\rho<0$, $u_{\pm\epsilon}$ satisfies \eqref{localsolution1} with $f$ replaced with solution $f_{C\pm\epsilon}$ 
of the problem \eqref{selfsimilarODE} with $C$ replaced with $C\pm\epsilon$. Due to continuous dependence of $f$ and $\xi_*$ on $C$, from \eqref{local3}, \eqref{localsolution1} for $u$ easily follows.\\ 
(ii) \& (iii)\ Assume that the condition of the case (ii) or (iii) with $b>0$ is fulfilled. As before, from \eqref{CauchyProblem3}, \eqref{local1}-\eqref{local3} follows.
Changing the variable $y=x+\bar{x}$ with $\bar{x}=(R,0,....,0)$, the function
$v_{\pm\epsilon}(y,t)=u_{\pm\epsilon}(y-\bar{x},t)$ solves the problem
\begin{gather}
v_t-\Delta v^{m}+bv^{\beta}=0,~y\in\mathbb{R}^N,~ t>0\label{changedproblem1}\\
v(y,0)= (C\pm\epsilon)(R-|y-\bar{x}|)_+^{\alpha}, \ y\in \mathbb{R}^N\label{changedproblem2}
\end{gather} 
Rescaled function 
\begin{equation}\label{rescaling}
v_{\pm\epsilon}^k(y,t)=kv_{\pm\epsilon}(k^{-\frac{1}{\alpha}}y,k^\frac{\alpha(m-1)-2}{\alpha}t)
\end{equation}
solves the problem
 \[\left\{
  \begin{array}{ll}
w_t-\Delta w^m+k^{\frac{\alpha(m-\beta)-2}{\alpha}}bw^{\beta}=0,\ y\in \mathbb{R}^N,\ t>0. \\
w(y,0)=(C\pm\epsilon)(k^{\frac{1}{\alpha}}R-|y-k^{\frac{1}{\alpha}}\bar{x}|)^{\alpha}_+, \ y\in \mathbb{R}^N \\
 \end{array}
 \right.\]
 From the comparison Theorem~\ref{comparison} and \eqref{limitu0} it follows that the sequence $\{v_{\pm\epsilon}^k\}$ is uniformly bounded
 by the solution of the CP \eqref{CauchyProblem1},\eqref{CauchyProblem2} with $b=0$ and $u_0(x)=C(x_1)^\alpha_+$. Since $\alpha(m-\beta)-2<0$,
 it easily follows that the sequence $\{v_{\pm\epsilon}^k\}$ converges to the solution of the CP \eqref{CauchyProblem1},\eqref{CauchyProblem2} with $b=0$ and $u_0(x)=C(x_1)^\alpha_+$. The rest of the proof coincides with the one given in case (i) above.\\
Finally, consider the case (iii) with $b<0$. Let $u_{\pm\epsilon}$ be a solution of the problem
\begin{gather}
u_t-\Delta u^{m}+bu^\beta=0,~|x|<2R,~ 0<t<\delta\label{localizedeq1}\\
u(x,0)=(C\pm\epsilon)(R-|x|)_+^{\alpha},\ |x|\leq 2R\label{localizedeq2}\\
u|_{|x|=2R}=0,\ 0\leq t \leq \delta.\label{localizedeq3}
\end{gather}
Due to finite speed of propagation property, the solution of the CP \eqref{CauchyProblem1},\eqref{CauchyProblem2} will vanish as $|x|=2R, 0\leq t \leq \delta$ for some $\delta>0$. 
Therefore, by comparison theorem we have \eqref{local2} for $R_\epsilon \leq |x| \leq 2R, 0\leq t \leq\delta_\epsilon$. Now, the function $v_{\pm\epsilon}(y,t)=u_{\mp\epsilon}(y-\bar{x},t)$ solves the problem
\begin{gather}
v_t-\Delta v^{m}+bv^{\beta}=0,~|y-\bar{x}|<2R,~ 0<t<\delta\label{transformedloc1}\\
v(y,0)= (C\pm\epsilon)(R-|y-\bar{x}|)_+^{\alpha}, \ |y-\bar{x}|<2R\label{transformedloc2}\\
v|_{|y-\bar{x}|=2R}=0,~0\leq t \leq \delta\label{transformedloc3}
\end{gather} 
and the function $w=v_{\pm\epsilon}^k$ rescaled as in \eqref{rescaling}, 
solves the problem
 \begin{gather}
L_kw\equiv w_t-\Delta w^m+k^{\frac{\alpha(m-\beta)-2}{\alpha}}bw^{\beta}=0,\ |y-k^{\frac{1}{\alpha}}\bar{x}|<2Rk^{\frac{1}{\alpha}},\ 0<t<k^{\frac{2-\alpha(m-1)}{\alpha}}\delta,\label{rescaledeq1}\\
w(y,0)=(C\pm\epsilon)(k^{\frac{1}{\alpha}}R-|y-k^{\frac{1}{\alpha}}\bar{x}|)^{\alpha}_+, \ |y-k^{\frac{1}{\alpha}}\bar{x}|<2Rk^{\frac{1}{\alpha}},\label{rescaledeq2}\\
w|_{|y-k^{\frac{1}{\alpha}}\bar{x}|=2Rk^{\frac{1}{\alpha}}}=0,\ 0\leq t \leq k^{\frac{2-\alpha(m-1)}{\alpha}}\delta. \label{rescaledeq3}
 \end{gather}
To prove the convergence of the sequence $\{v_{\pm\epsilon}^k\}$ we first prove the uniform boundedness. Consider a function
\[ g(x,t)=(C+1)(1+|y|^2)^{\frac{\alpha}{2}} (1-\mu t)^{\frac{1}{1-m}}, \ y\in\mathbb{R}^N, \ 0\leq t \leq t_0\coloneqq\mu^{-1}/2, \]
with
\begin{gather}
\mu=\bar{H}+1, \bar{H}=\max\limits_{y\in\mathbb{R}^N}H(y),\nonumber\\
 H(y)=\alpha m(m-1)(C+1)^{m-1} \max(1; N+2-\alpha m) \Big(1+|y|^2\Big)^{\frac{\alpha(m-1)-2}{2}}.\nonumber
 \end{gather}
We have
\begin{gather}
L_kg=(m-1)^{-1}(C+1)(1+|y|^2)^{\frac{\alpha}{2}}(1-\mu t)^{\frac{m}{1-m}} S,\  S\geq 1+R\nonumber\\
R=b(m-1)(C+1)^{\beta-1}k^{\frac{\alpha(m-\beta)-2}{\alpha}}(1+|y|^2)^{\frac{\alpha(\beta-1)}{2}}(1-\mu t)^{\frac{\beta-m}{1-m}},\nonumber
\end{gather}
and
\[ R=O\Big (k^{\frac{\alpha(m-1)-2}{\alpha}}\Big),\ \text{uniformly for} \ |y-k^{\frac{1}{\alpha}}\bar{x}|\leq 2Rk^{\frac{1}{\alpha}},\ 0\leq t\leq t_0, \ \text{as} \ k\to+\infty \]
We also have
\begin{equation}\label{comparisononparabolicb}
g(y,0)\geq v_{\pm\epsilon}^k(y,0),\ |y-k^{\frac{1}{\alpha}}\bar{x}|<2Rk^{\frac{1}{\alpha}}; \ g|_{|y-k^{\frac{1}{\alpha}}\bar{x}|=2Rk^{\frac{1}{\alpha}}}\geq 0.
\end{equation}
Therefore, for all sufficiently large $k$ $g$ is a supersolution of the problem \eqref{rescaledeq1}-\eqref{rescaledeq3}. From the Theorem \ref{comparison} it follows that
\begin{equation}\label{uniformbound}
 0\leq v_{\pm\epsilon}^k(y,t)\leq g(y,t),\ |y-k^{\frac{1}{\alpha}}\bar{x}|\leq 2Rk^{\frac{1}{\alpha}},\ 0\leq t \leq t_0. 
 \end{equation}
Hence, the sequence $\{ v_{\pm\epsilon}^k\}$ is uniformly bounded in a strip $\{0\leq t\leq t_0\}$. Standard regularity result for the nonlinear degenerate parabolic equations \cite{Dibe-Sv}
imply that the sequence is uniformly H\"{o}lder continuous on compact subsets of $\{0<t\leq t_0\}$. Arzela-Ascoli theorem and standard diagonalization argument imply that there is a pointwise convergent subsequence in $\{0<t\leq t_0\}$, with uniform convergence on compact subsets. Since, $\alpha(m-\beta)<2$ it easily follows that the limit function is a solution 
of the CP \eqref{CauchyProblem1},\eqref{CauchyProblem2} with $b=0$ and $u_0(x)=C(x_1)^\alpha_+$. Due to uniqueness of the latter, the whole sequence converges to its unique limit point, and the rest of the proof is completed as in previous cases.
\end{proof}
\begin{lem}\label{lemma3}
Let $b>0,\ 0<\beta<1, \alpha=\frac{2}{m-\beta}$. Then solution $u$ of the CP \eqref{CauchyProblem1}-\eqref{CauchyProblem3} satisfies \eqref{asympsolution2}.
\end{lem}
\begin{proof}[Proof] As before, from \eqref{CauchyProblem3} we deduce \eqref{local1}-\eqref{local3} in the context of this lemma. 
Changing the variable $y=x+\bar{x}$ with $\bar{x}=(R,0,....,0)$, the function
$v_{\pm\epsilon}(y,t)=u_{\pm\epsilon}(y-\bar{x},t)$ solves the problem \eqref{changedproblem1},\eqref{changedproblem2} with $\alpha=2/(m-\beta)$.
Rescaled function 
\begin{equation}\label{rescaling1}
v_{\pm\epsilon}^k(y,t)=kv_{\pm\epsilon}(k^{\frac{\beta-m}{2}}y,k^{\beta-1}t)
\end{equation}
solves the problem
 \[\left\{
  \begin{array}{ll}
w_t-\Delta w^m+bw^{\beta}=0,\ y\in \mathbb{R}^N,\ t>0. \\
w(y,0)=(C\pm\epsilon)(k^{\frac{m-\beta}{2}}R-|y-k^{\frac{m-\beta}{2}}\bar{x}|)^{\frac{2}{m-\beta}}_+, \ y\in \mathbb{R}^N \\
 \end{array}
 \right.\]
Since \eqref{limitu0} is valid with $\alpha=2/(m-\beta)$, the limit of the sequence $\{v_{\pm\epsilon}^k\}$ solves the 
CP \eqref{CauchyProblem1},\eqref{CauchyProblem2} with $b=0$ and $u_0(x)=C(x_1)^{\frac{2}{m-\beta}}_+$.    
Due to uniqueness of the solution to the CP (\cite{KPV}), the latter coincides with the solution of the 1D CP  \eqref{1dreactiondiffusion},\eqref{1dreactiondiffusionic}, which is of self-similar form \eqref{selfsimilarform1} with the shape function $h(\zeta)$ solving nonlinear ODE problem \eqref{selfsimilarODE2} and having finite interface $\zeta_*$ \cite{Abdulla1}. Therefore, we have
\[\lim_{k \rightarrow \infty}k v(k^{\frac{\beta-m}{2}}y,k^{\beta-1}t)=t^{\frac{1}{1-\beta}}h\Big(y_1t^{-\frac{m-\beta}{2(1-\beta)}}\Big).\]
The remainder of the proof of \eqref{asympsolution2} proceeds similar to the proof of  \eqref{localsolution1} in Lemma~\ref{lemma1} (i) and (ii). 
In particular, if $C>C_*$ we have\ $\zeta_*<0$ and for $\forall \rho \in (\zeta_*,0)$
\begin{equation}\label{equivalentasymp-1}
u(x,t)\sim h(\rho)  t^{\frac{1}{1-\beta}},\ t=\eta(x)=\Big(\frac{R-|x|}{\rho}\Big)^{\frac{2(1-\beta)}{m-\beta}}~as~|x|\rightarrow R+,
\end{equation}
while If $C<C_*$ we have $\zeta_*>0$ and for $\forall \rho >\zeta_*$
\begin{equation}\label{equivalentasymp-2}
u(x,t)\sim h(\rho)  t^{\frac{1}{1-\beta}},\ t=\eta(x)=\Big(\frac{R-|x|}{\rho}\Big)^{\frac{2(1-\beta)}{m-\beta}}\ as\ |x|\rightarrow R-.
\end{equation}
\end{proof}
\begin{lem}\label{lemma4}
Let $b>0,\ 0<\beta<1, \alpha>\frac{2}{m-\beta}$. Then solution $u$ of the CP \eqref{CauchyProblem1}-\eqref{CauchyProblem3} satisfies \eqref{shrinkingsolution}.
\end{lem}
\begin{proof}[Proof] As in the proof of Lemma~\ref{lemma1}, case (iii) we set \eqref{localizedeq1}-\eqref{localizedeq3}, deduce \eqref{local2} for $R_\epsilon \leq |x| \leq 2R, 0\leq t \leq\delta_\epsilon$, and derive the transformed problem \eqref{transformedloc1}-\eqref{transformedloc3} in the context of this lemma.
Rescaled solution according to invariant scale for reaction equation
\begin{equation}\label{rescaling1}
v_{\pm\epsilon}^k(y,t)=kv_{\pm\epsilon}(k^{-\frac{1}{\alpha}}y,k^{\beta-1}t)
\end{equation}
solves the problem
 \begin{gather}
\tilde{L}_kw\equiv w_t-k^{\frac{2-\alpha(m-\beta)}{\alpha}}\Delta w^m+bw^{\beta}=0,\ |y-k^{\frac{1}{\alpha}}\bar{x}|<2Rk^{\frac{1}{\alpha}},\ 0<t<k^{1-\beta}\delta,\label{rescaledreaction1}\\
w(y,0)=(C\pm\epsilon)(k^{\frac{1}{\alpha}}R-|y-k^{\frac{1}{\alpha}}\bar{x}|)^{\alpha}_+, \ |y-k^{\frac{1}{\alpha}}\bar{x}|<2Rk^{\frac{1}{\alpha}},\label{rescaledreaction2}\\
w|_{|y-k^{\frac{1}{\alpha}}\bar{x}|=2Rk^{\frac{1}{\alpha}}}=0,\ 0\leq t \leq k^{1-\beta}\delta. \label{rescaledreaction3}
 \end{gather}
To prove the uniform boundedness of $\{v_{\pm\epsilon}^k\}$ consider a function
\[ g(x,t)=(C+1)(1+|y|^2)^{\frac{\alpha}{2}}e^t, \ y\in\mathbb{R}^N, \ 0\leq t \leq T, \]
for some fixed $T>0$. We have
\begin{gather}
\tilde{L}_kg\geq g(1-\Gamma),\nonumber\\
 \Gamma=(C+1)^m\alpha m e^{(m-1)t}(1+|y|^2)^{\frac{\alpha(m-1)-4}{2}}(N+(\alpha m +N-2)|y|^2) k^{\frac{2-\alpha(m-\beta)}{\alpha}}\label{Gamma}\\
 \Gamma=O(k^\gamma),\ \text{uniformly for} \ |y-k^{\frac{1}{\alpha}}\bar{x}|\leq 2Rk^{\frac{1}{\alpha}},\ 0\leq t\leq T, \ \text{as} \ k\to+\infty,\nonumber
 \end{gather}
where,
\[
    \gamma=\left\{
                \begin{array}{ll}
                  {\frac{2-\alpha(m-\beta)}{\alpha}},\ if \ \alpha<{\frac{2}{m-1}}\\
                  \beta-1,\qquad if\quad \alpha\geq{\frac{2}{m-1}}\\
                  \end{array}
              \right.
  \]
The estimation \eqref{comparisononparabolicb} is clearly satisfied. Therefore, for sufficiently large $k$, $g$ is a supersolution of \eqref{rescaledreaction1}-\eqref{rescaledreaction3}
and \eqref{uniformbound} is true in this context in $0\leq t \leq T$. The proof of the convergence of the sequence $\{v_{\pm\epsilon}^k\}$, and desired estimation \eqref{shrinkingsolution}
is completed as in the proof case (iii) of Lemma~\ref{lemma1}. \end{proof}

\section{Proofs of the main results.}\label{proofs}
 In this section we prove the main results described in section 2.
\begin{proof}[Proof of Theorem~\ref{theorem1}] The estimation \eqref{localsolution1}, and its equivalent \eqref{equiv 6} are proved in 
Lemma~\ref{lemma1}. They imply that $\eta_+$ is defined and finite, and 
\begin{equation}\label{eta+limsup}
\limsup_{|x|\rightarrow R+}\frac{\eta_+(x)}{(|x|-R)^{2+\alpha(1-m)}}\leq (-\xi_*)^{\alpha(m-1)-2}.
\end{equation}
As before, we deduce \eqref{local1}-\eqref{local3} from \eqref{CauchyProblem3}, and consider the problem \eqref{changedproblem1}-\eqref{changedproblem2} for $v_\epsilon(y,t)=u_\epsilon(y-\bar{x},t)$.
Let $w_\epsilon(y,t)=w_\epsilon(y_1,t)$ be a solution of the Cauchy problem  \eqref{1ddiffusion},\eqref{1ddiffusionic} with $C$ replaced by $C+\epsilon$. Assume that $b\geq 0$. Since
\begin{equation}\label{comparisoninitialfunction}
(y_1)^{\alpha}_+\geq (R-|y-\bar{x}|)_+^{\alpha},~y\in\mathbb{R}^N,
\end{equation}
from the comparison theorem it follows that
\begin{equation}\label{vC+epsilon2}
0\leq v_\epsilon(y,t)\leq w_\epsilon(y_1,t),~y\in \mathbb{R}^N, t>0.
\end{equation}
From  \eqref{selfsimilarform}-\eqref{CP6}it follows that
\begin{equation}\label{vC+epsilon3}
v_\epsilon(y,t)=0, \qquad{for} \ y_1\leq (C+\epsilon)^{\frac{m-1}{2-\alpha(m-1)}}\xi_*^{'} t^{\frac{1}{2+\alpha(1-m)}},~t\geq 0
\end{equation}
that is to say,
\begin{equation}\label{uC+epsilon}
u_\epsilon(x,t)=0, \qquad{for} \ x_1\leq -R+(C+\epsilon)^{\frac{m-1}{2-\alpha(m-1)}}\xi_*^{'} t^{\frac{1}{2+\alpha(1-m)}},~t\geq 0
\end{equation}
From \eqref{local1}-\eqref{local3} it follows that for arbitrary $\epsilon>0$ $\exists 0<R_\epsilon<R$ and $\delta_\epsilon$ such that
\begin{equation}\label{uC+epsilon1}
u(x,t)=0, \qquad{for} \ x_1\leq -R+(C+\epsilon)^{\frac{m-1}{2-\alpha(m-1)}}\xi_*^{'} t^{\frac{1}{2+\alpha(1-m)}},~0\leq t\leq \delta_\epsilon.
\end{equation}
Due to radial symmetricity of $u$ we have
\begin{equation}\label{eta+estimate1}
u(x,t)=0,  \qquad{for} \ |x|\geq R-(C+\epsilon)^{\frac{m-1}{2-\alpha(m-1)}}\xi_*^{'} t^{\frac{1}{2+\alpha(1-m)}},~0\leq t\leq \delta_\epsilon.
\end{equation}
This implies that for some $\gamma_\epsilon >0$ we have
\begin{equation}\label{eta+estimate2}
\eta_+(x)\geq \Big ( \frac{R-|x|}{(C+\epsilon)^{\frac{m-1}{2-\alpha(m-1)}}\xi_*^{'}}\Big)^{2+\alpha(1-m)}~\mathrm{for}~R<|x|\leq R+\gamma_\epsilon
\end{equation}
Therefore, we have
\[\liminf_{|x|\rightarrow R+}\frac{\eta_+(x)}{(|x|-R)^{2+\alpha(1-m)}}\geq (-(C+\epsilon)^{\frac{m-1}{2-\alpha(m-1)}}\xi_*^{'})^{\alpha(m-1)-2}.\]
By taking~$\epsilon \downarrow 0$ we have
\begin{equation}\label{eta+liminf}
\liminf_{|x|\rightarrow R+}\frac{\eta_+(x)}{(|x|-R)^{2+\alpha(1-m)}}\geq (-\xi_*)^{\alpha(m-1)-2}.
\end{equation}
From \eqref{eta+limsup} and \eqref{eta+liminf} desired estimation \eqref{eta+} follows if $b\geq 0$. If $b<0, \ \beta\geq 1$, we consider a function
\[ \bar{u}_\epsilon=exp(-bt)u_\epsilon \Big ( x,\Big( b(1-m)\Big )^{-1} \Big [ exp\Big (b(1-m)t\Big )-1\Big ]\Big ) \] 
where $u_\epsilon$ is a solution of the CP \eqref{CauchyProblem1}, \eqref{CauchyProblem4} with $b=0$ and $C+\epsilon$. Accordingly, $\bar{u}_\epsilon$ solves the CP \eqref{CauchyProblem1}, \eqref{CauchyProblem4} with $\beta=1$ and $C+\epsilon$. 
By continuity of solution we can choose $\delta_\epsilon>0$, $0<R_\epsilon<R$ such that 
\[ \bar{u}_\epsilon <1 \ \qquad{for} \ |x|>R_\epsilon, 0<t<\delta_\epsilon. \]
Therefore, $\bar{u}_\epsilon$ is a supersolution of the PDE \eqref{CauchyProblem1} with $b<0, \ \beta \geq 1$. 
Similar arguments used in the derivation of \eqref{eta+estimate1} imply that
\begin{equation}\label{eta+estimate b negative}
u(x,t)=0,  \qquad{for} \ |x|\geq R-(C+\epsilon)^{\frac{m-1}{2-\alpha(m-1)}}\xi_*^{'} \Big (\frac{e^{b(1-m)t)}-1}{b(1-m)} \Big )^{\frac{1}{2+\alpha(1-m)}},~0\leq t\leq \delta_\epsilon.
\end{equation}
Therefore, for some $\gamma_\epsilon >0$ we have
\begin{equation}\label{eta+estimate3}
\eta_+(x)\geq \frac{1}{b(1-m)} \log \Big [ 1+b(1-m) \Big (\frac{R-|x|}{(C+\epsilon)^{\frac{m-1}{2-\alpha(m-1)}}\xi_*^{'}}\Big)^{2+\alpha(1-m)}\Big ], \ R<|x|\leq R+\gamma_\epsilon
\end{equation}
Passing to {\it liminf} as $|x|\to R+$, and then passing to limit as $\epsilon \to 0$, \eqref{eta+liminf} follows. As before, 
from \eqref{eta+limsup} and \eqref{eta+liminf} desired estimation \eqref{eta+} again follows.
\end{proof}
\begin{proof}[Proof of Theorem~\ref{theorem3}] Asymptotic estimation \eqref{asympsolution2}, and its equivalents \eqref{equivalentasymp-1}, \eqref{equivalentasymp-2} are proved in Lemma~\ref{lemma3}. If $C>C_*$, from  \eqref{equivalentasymp-1} it follows that $\eta_+$ is defined and finite, and 
\begin{equation}\label{eta+limsup-1} 
\limsup_{|x|\rightarrow R+}\frac{\eta_+(x)}{\Big(|x|-R\Big)^{\frac{2(1-\beta)}{m-\beta}}}\leq (-\zeta_*)^{\frac{2(1-\beta)}{\beta-m}}.
\end{equation}
Similarly, if $C<C_*$, from  \eqref{equivalentasymp-2} it easily follows that
\begin{equation}\label{eta-liminf-1}
\liminf_{|x|\rightarrow R-}\frac{\eta_-(x)}{\Big(R-|x|\Big)^{\frac{2(1-\beta)}{m-\beta}}}\geq (\zeta_*)^{\frac{2(1-\beta)}{\beta-m}}.
\end{equation}
First, consider the global case of initial function \eqref{CauchyProblem4}. Changing the variable $y=x+\bar{x}$ with $\bar{x}=(R,0,....,0)$, the function
$v(y,t)=u(y-\bar{x},t)$ solves the problem \eqref{changedproblem1}-\eqref{changedproblem2}  with $\alpha=2/(m-\beta)$ and $\epsilon=0$.
As before, from \eqref{comparisoninitialfunction} with $\alpha=2/(m-\beta)$ and comparison theorem, \eqref{vC+epsilon2} with $\epsilon=0$ follows. In our context, $w(y_1,t)$ is a unique solution
of the CP \eqref{1dreactiondiffusion},\eqref{1dreactiondiffusionic}, which is of self-similar form \eqref{selfsimilarform1} with the shape function $h$ solving nonlinear ODE problem \eqref{selfsimilarODE2}, and having a finite interface $\zeta_*$. If $C>C_*$ from \cite{Abdulla1} it follows that
\begin{equation}\label{VW2}
w(y_1,t)\leq  C_1t^{\frac{1}{1-\beta}}(-\zeta_1+\zeta)^{\frac{2}{m-\beta}}_+\ for\ -\infty<y_1\leq 0,\ 0\leq t < +\infty.
\end{equation}
(see Appendix for explicit values of the constants $C_1, \zeta_1$). From \eqref{vC+epsilon2} with $\epsilon=0$ we have
\begin{equation}\label{C>C_*expansion}
u(x,t)\leq  C_1\Big(x_1+R-\zeta_1t^{\frac{m-\beta}{2(1-\beta)}}\Big)^{\frac{2}{m-\beta}}_+,\   -\infty<x_1\leq-R, \ 0\leq t<+\infty.
\end{equation}
Due the radial symmetricity of $u$ from \eqref{C>C_*expansion} we deduce that
\begin{equation}\label{C>C_*expansion-1}
u(x,t)\leq  C_1\Big(R-|x|-\zeta_1t^{\frac{m-\beta}{2(1-\beta)}}\Big)^{\frac{2}{m-\beta}}_+,\   |x|\geq R, \ 0\leq t<+\infty,
\end{equation}
which imply
\[\eta_+(x)\geq\Big(\frac{|x|-R}{-\zeta_1}\Big)^{\frac{2(1-\beta)}{m-\beta}},\ |x|\geq R\]
and therefore,
\begin{equation}\label{eta_+lower2}
\liminf_{|x|\rightarrow R+}\frac{\eta_+(x)}{(|x|-R)^{\frac{2(1-\beta)}{m-\beta}}}\geq \big(-\zeta_1\big)^{\frac{2(1-\beta)}{\beta-m}}
\end{equation}
From \eqref{eta+limsup-1} and \eqref{eta_+lower2}, \eqref{eta_+1} follows.

Assume that $0<C<C_*$. From \cite{Abdulla1} it follows that if $m+\beta>2$, then
\[0\leq w(y_1,t)\leq \Big[C^{1-\beta}(y_1)^{\frac{2(1-\beta)}{m-\beta}}_+-b(1-\beta)\big(1-(C/C_*)^{m-\beta}\big)t\Big]^{\frac{1}{1-\beta}}_+,\ y_1\in \mathbb{R}, \ 0\leq t < +\infty\]
 From \eqref{vC+epsilon2} with $\epsilon=0$ we have
 \[0\leq u(x,t)\leq \Big[C^{1-\beta}(x_1+R)^{\frac{2(1-\beta)}{m-\beta}}_+-b(1-\beta)\big(1-(C/C_*)^{m-\beta}\big)t\Big]^{\frac{1}{1-\beta}}_+,\ x\in \mathbb{R}^N, \ 0\leq t < +\infty.\]
 Due the radial symmetricity of $u$ it follows that
 \begin{equation}\label{m=beta,2-2}
 0\leq u(x,t)\leq \Big[C^{1-\beta}(R-|x|)^{\frac{2(1-\beta)}{m-\beta}}_+-b(1-\beta)\big(1-(C/C_*)^{m-\beta}\big)t\Big]^{\frac{1}{1-\beta}}_+,\ x\in \mathbb{R}^N, \ t\geq 0.\
 \end{equation}
 If\ $1\leq m<2-\beta$,\ then from \cite{Abdulla1} it follows that
\[0\leq w(y_1,t)\leq C_2\Big(y_1-\zeta_2t^{\frac{m-\beta}{2(1-\beta)}}\Big)^{\frac{2}{m-\beta}}_+,\ y_1\leq l_1t^{\frac{m-\beta}{2(1-\beta)}}, 0\leq t <+\infty\]
(the values of the constants $C_2>0, 0<\zeta_2<l_1$ are given in Appendix). From \eqref{vC+epsilon2} with $\epsilon=0$ we have
\[0\leq u(x,t)\leq C_2\Big(x_1+R-\zeta_2t^{\frac{m-\beta}{2(1-\beta)}}\Big)^{\frac{2}{m-\beta}}_+,\ x_1\leq -R+l_1t^{\frac{m-\beta}{2(1-\beta)}}, 0\leq t <+\infty\]
and due to radial symmetricity of the solution $u$ it follows that
\begin{equation}\label{upper11}
0\leq u(x,t)\leq C_2\Big(R-|x|-\zeta_2t^{\frac{m-\beta}{2(1-\beta)}}\Big)^{\frac{2}{m-\beta}}_+,\ |x|\geq R-l_1t^{\frac{m-\beta}{2(1-\beta)}}, 0\leq t <+\infty.
\end{equation}
From \eqref{m=beta,2-2} and \eqref{upper11} it follows that for some $\mu>0$
\[\eta_-(x)\leq \Big(\frac{R-|x|}{\bar{\zeta}}\Big)^{\frac{2(1-\beta)}{m-\beta}}, \ R-\mu \leq |x| <R,\]
which imply
\begin{equation}\label{limsupshrinkingcase}
\limsup_{|x|\rightarrow R-}\frac{\eta_-(x)}{\Big(R-|x|\Big)^{\frac{2(1-\beta)}{m-\beta}}}\leq (\zeta_2)^{\frac{2(1-\beta)}{\beta-m}}.
\end{equation}
From \eqref{eta-liminf-1} and \eqref{limsupshrinkingcase}, \eqref{eta_-1} follows.

In the local case when initial condition satisfies \eqref{CauchyProblem3}, we first deduce \eqref{local1}-\eqref{local3} in the context of this theorem, and then apply the presented proof
to $u_{\pm}$ and subsequently pass to limit as $\epsilon \to 0$. 

Note that in the special case $m+\beta=2$ as in Corollary~\ref{corollary}, $w_\epsilon(y_1,t)$ is a unique solution
of the CP \eqref{1dreactiondiffusion},\eqref{1dreactiondiffusionic} with\ $C$\ replaced by \ $C+\epsilon$ given as follows
\[w_\epsilon(y_1,t)= (C+\epsilon)(y_1-\zeta^\epsilon_*t)^{\frac{1}{1-\beta}}\]
where $\zeta^\epsilon_*$ is defined by \eqref{m+beta=2} with $C$ replaced by $C+\epsilon$.
\end{proof}
\begin{proof}[Proof of Theorem~\ref{theorem4}] Asymptotic estimation \eqref{shrinkingsolution} is proved in Lemma~\ref{lemma4}. It implies that for any $l>l_*$ there exists $\gamma_l>0$ such that 
  \[\eta_-(x)\geq \Big(\frac{R-|x|}{l}\Big)^{\alpha(1-\beta)},\quad R-\gamma_l\leq |x|<R.\]
Passing to $\liminf$ as $|x|\to R-$, followed by limit as $l\to l_*$, we have  
\begin{equation}\label{eta_-lower4}
\liminf_{|x|\rightarrow R-} \frac{\eta_-(x)}{(R-|x|)^{\alpha(1-\beta)}}\geq \big(l_*\big)^{\alpha(\beta-1)}
\end{equation}
To prove the opposite inequality, first from \eqref{CauchyProblem3} we deduce \eqref{local1}-\eqref{local3} in the context of this theorem. Changing the variable $y=x+\bar{x}$ with $\bar{x}=(R,0,....,0)$, the function
$v_{\pm\epsilon}(y,t)=u_{\pm\epsilon}(y-\bar{x},t)$ solves the problem \eqref{changedproblem1}-\eqref{changedproblem2}  with $\alpha>2/(m-\beta)$.
Let $w_\epsilon(y,t)$ be a solution of the Cauchy-Dirichlet problem for the PDE \eqref{changedproblem1} in 
\[ \mathcal{D}_\delta=\{(y,t)\in \mathbb{R}^N\times (0,\delta]: y_1 < R-R_\epsilon\} \]
under the conditions:
\begin{gather}
w(y,0)=(C+2\epsilon) (y_1)_+^\alpha, \ -\infty <y_1\leq R-R_\epsilon, (y_2,...,y_N)\in \mathbb{R}^{N-1}\label{changedproblemic-3}\\
w\Big |_{y_1=R-R_\epsilon}=(C+2\epsilon)(R-R_\epsilon), \ (y_2,...,y_N)\in \mathbb{R}^{N-1}, \ 0\leq t \leq \delta.\label{changedproblembc-3}
\end{gather}
From \eqref{comparisoninitialfunction} it follows that
\begin{equation}\label{comparison-4}
v_\epsilon(y,0)\leq w_\epsilon(y,0), \ -\infty <y_1\leq R-R_\epsilon, (y_2,...,y_N)\in \mathbb{R}^{N-1}.
\end{equation}
Due to finite speed of propagation property and continuity of $v_\epsilon$ in $\overline{\mathcal{D}}_\delta$ it follows that for some $\delta_\epsilon>0$ we have
\begin{equation}\label{comparisonboundary}
v_\epsilon\Big|_{y_1=R-R_\epsilon} \leq (C+\epsilon)(R-R_\epsilon) \leq w_\epsilon\Big|_{y_1=R-R_\epsilon}, \ 0\leq t \leq \delta_\epsilon.
\end{equation}
From \eqref{comparisoninitialfunction},\eqref{comparisonboundary} and comparison Theorem~\ref{comparison} it follows that
\begin{equation}\label{compshrinkingcase}
v_\epsilon(y,t)\leq w_\epsilon(y,t), \ (y,t)\in\overline{\mathcal{D}}_{\delta_\epsilon}.
\end{equation}
Due to uniqueness of the solution to the Cauchy-Dirichlet problem  \eqref{changedproblem1},\eqref{changedproblemic-3},\eqref{changedproblembc-3} in $\mathcal{D}_{\delta_\epsilon}$,
we have $w_\epsilon(y,t)=w_{1\epsilon}(y_1,t)$, where the latter is a unique solution of the one-dimensional Cauchy-Dirichlet problem 
\begin{gather}
w_t-w^m_{y_1y_1}+bw^\beta=0, \ -\infty<y_1<y_{1\epsilon}\coloneqq R-R_\epsilon, 0 < t <\delta_\epsilon\label{cd-1}\\
w(y_1,0)=(C+2\epsilon) (y_1)_+^\alpha, \ -\infty <y_1\leq y_{1\epsilon}\label{cd-2}\\
w({y_{1\epsilon}},t)=(C+2\epsilon)y_{1\epsilon}, \ 0\leq t \leq \delta_\epsilon.\label{cd-3}
\end{gather}
From \cite{Abdulla1} it follows that if $m+\beta\geq 2$ then
\[0\leq w_{1\epsilon}(y_1,t) \leq [(C+3\epsilon)^{1-\beta}(y_1)^{\alpha(1-\beta)}_+-b(1-\beta)(1-\epsilon)t)^{\frac{1}{1-\beta}}]_+, \ -\infty<y_1\leq y_{1\epsilon}, 0\leq t \leq\delta_\epsilon. \]
Therefore, from \eqref{local3},\eqref{compshrinkingcase} it follows that
\[0 \leq u(x,t) \leq ((C+3\epsilon)^{1-\beta}(x_1+R)^{\alpha(1-\beta)}_+-b(1-\beta)(1-\epsilon)t)^{\frac{1}{1-\beta}}_+, \ -\infty<x_1\leq -R_\epsilon, 0\leq t \leq \delta_\epsilon\]
and due to the radial symmetricity of the solution $u$ we deduce the estimation
\[0\leq u(x,t) \leq \Big \{(C+3\epsilon)^{1-\beta}(R-|x|)^{\alpha(1-\beta)}_+-b(1-\beta)(1-\epsilon)t\Big\}^{\frac{1}{1-\beta}}_+,\ |x|\geq R_\epsilon, 0\leq t \leq \delta_\epsilon.\]
Therefore, we have
\[\eta_-(x)\leq \frac{(C+3\epsilon)^{1-\beta}(R-|x|)^{\alpha(1-\beta)}}{b(1-\beta)(1-\epsilon)}, \ R_\epsilon \leq |x|<R.\]
Taking $\limsup$ as $|x|\to R-$, followed by the limit as $\epsilon \downarrow 0$ we derive
\begin{equation}\label{eta_-upper4}
\limsup_{|x|\rightarrow R-} \frac{\eta_-(x)}{(R-|x|)^{\alpha(1-\beta)}}\leq \big(l_*\big)^{\alpha(\beta-1)}
\end{equation}
From \eqref{eta_-lower4} and \eqref{eta_-upper4}, \eqref{shrinkinginterface} follows.

If\ $1\leq m+\beta<2$\ from \cite{Abdulla1} it follows that for arbitrary $l>l_*$ and for all sufficiently small $\epsilon>0$ there exists $\delta_\epsilon(l)>0$ such that the solution of the Cauchy-Dirichlet problem \eqref{cd-1}-\eqref{cd-3} satisfies the following estimation:
\begin{equation}\label{compshrinkingcase-2}
0\leq w_\epsilon(y_1,t) \leq C_3\Big(y_1-\zeta_3t^{\frac{1}{\alpha(1-\beta)}}\Big)^{\alpha}_+, \ -\infty<y_1\leq lt^{\frac{1}{\alpha(1-\beta)}}, 0\leq t\leq \delta_\epsilon,
\end{equation}
where
\[ \zeta_3=\big(l_*/l\big)^{\alpha(1-\beta)}(1-\epsilon)l,\ C_3=\Big[1-(l_*/l\big)^{\alpha(1-\beta)}(1-\epsilon)\Big]^{-\alpha}\Big[C^{1-\beta}-l^{-\alpha(1-\beta)}b(1-\beta)(1-\epsilon)\Big]^{\frac{1}{1-\beta}}. \]
 From \eqref{local3},\eqref{compshrinkingcase-2} it follows that
\[0 \leq u(x,t) \leq C_3\Big(x_1+R-\zeta_3t^{\frac{1}{\alpha(1-\beta)}}\Big)^{\alpha}_+, \ -\infty<x_1\leq -R+lt^{\frac{1}{\alpha(1-\beta)}}, 0\leq t\leq \delta_\epsilon,\]
and due to radial symmetricity of $u$ we derive the estimation
\begin{equation}\label{upperbound-shrinkingcase}
0\leq u(x,t) \leq C_3\Big(R-|x|-\zeta_3t^{\frac{1}{\alpha(1-\beta)}}\Big)^{\alpha}_+,\ |x|\geq R-lt^{\frac{1}{\alpha(1-\beta)}}, 0\leq t \leq \delta_\epsilon.
\end{equation}
From \eqref{upperbound-shrinkingcase} it follows that for some $\gamma_\epsilon>0$ we have
\[\eta_-(x)\leq \Big(\frac{R-|x|}{\zeta_3}\Big)^{\alpha(1-\beta)},\qquad R-\gamma_\epsilon\leq |x|\leq R\]
Taking $\limsup$ as $|x|\to R-$, followed by limits as $\epsilon\downarrow 0$ and $l\downarrow l_*$ we deduce
\begin{equation}\label{eta_-upper4}
\limsup_{|x|\rightarrow R-}\frac{\eta_-(x)}{(R-|x|)^{\alpha(1-\beta)}}\leq (l_*)^{\alpha(\beta-1)}.   
\end{equation}
From \eqref{eta_-lower4} and \eqref{eta_-upper4}, \eqref{shrinkinginterface} again follows.
\end{proof}
Theorem~\ref{theorem5} is proved through direct application of Theorem~\ref{comparison} to upper and lower bounds given respectively in estimations \eqref{stationary1}, \eqref{stationary2}, \eqref{stationary3}, \eqref{stationary4}, \eqref{stationary4"}.


\newpage
\textbf{Appendix.}\label{sec: appendix}
Here we bring explicit values of the constants used in Sections~\ref{sec:mainresult} and~\ref{proofs}.\\\\
\begin{gather*}
\zeta_1=
\begin{cases}                
-A_1^{\frac{m-1}{2}}\big(1+b(1-\beta)A_1^{\beta-1}\big)^{-\frac{1}{2}}\big(2m(m+\beta)(1-\beta)\big)^{\frac{1}{2}}(m-\beta)^{-1}&~\quad\text{if} \   m+\beta>2\\
-(A_1/C_*)^{\frac{m-\beta}{2}}&~\quad\text{if} \  1\leq m+\beta<2,
\end{cases}\\
C_1=
\begin{cases}                
A_1(-\zeta_1)^{\frac{2}{m-\beta}}&~\quad\text{if} \   m+\beta>2\\
C_*&~\quad\text{if} \    1\leq m+\beta<2,
\end{cases}\\
A_1=w(0,1)=h(0).\\
\zeta_2=
\begin{cases}   
C^{\frac{\beta-m}{2}}\big(b(1-\beta)\big(1-\big(C/C_*\big)^{m-\beta}\big)\big)^{\frac{m-\beta}{2(1-\beta)}} \quad \text{if}~ m+\beta>2\\
\delta_*\Gamma l_1\quad\text{if}~ m+\beta<2,
\end{cases}\\
l_1=C^{\frac{\beta-m}{2}}\Big[b(1-\beta)\Big(\delta_*\Gamma\Big)^{-1}\Big(1-\delta_*\Gamma-\Big(1-\delta_*\Gamma\Big)^{-1}\big(C/C_*\big)^{m-\beta}\Big)\Big]^{\frac{m-\beta}{2(1-\beta)}},\\
\Gamma=1-\big(C/C_*\big)^{\frac{m-\beta}{2}},\quad C_2=C\big(1-\delta_*\Gamma\big)^{\frac{2}{\beta-m}},\\
\quad\text{and $\delta_*\in(0,1)$ satisfies}\\
g(\delta_*)=\max_{[0,1]}g(\delta),\qquad g(\delta)={\delta}^{\frac{2-\beta-m}{m-\beta}}\Big[1-\delta \Gamma-\Big(1-\delta \Gamma\Big)^{-1}\big(C/C_*\big)^{m-\beta}\Big]
\end{gather*}

\end{document}